\documentclass[reqno, 12pt]{amsart}
\pdfoutput=1
\makeatletter
\let\origsection=\section \def\section{\@ifstar{\origsection*}{\mysection}} 
\def\mysection{\@startsection{section}{1}\z@{.7\linespacing\@plus\linespacing}{.5\linespacing}{\normalfont\scshape\centering\S}}
\makeatother 

\usepackage{amsmath,amssymb,amsthm}
\usepackage{mathrsfs}
\usepackage{mathabx}\changenotsign
\usepackage{dsfont}
 
\usepackage{xcolor}
\usepackage[backref]{hyperref}
\hypersetup{
    colorlinks,
    linkcolor={red!60!black},
    citecolor={green!60!black},
    urlcolor={blue!60!black}
}

\usepackage{graphicx}

\usepackage[open,openlevel=2,atend]{bookmark}

\usepackage[abbrev,msc-links,backrefs]{amsrefs} 
\usepackage{doi}

\renewcommand{\PrintDOI}[1]{\doi{#1}}

\usepackage[T1]{fontenc}
\usepackage{lmodern}
\usepackage[babel]{microtype}
\usepackage[english]{babel}

\usepackage{geometry}
\numberwithin{equation}{section}
\numberwithin{figure}{section}

\usepackage{enumitem}

\let\polishlcross=\l
\def\l{\ifmmode\ell\else\polishlcross\fi}

\makeatletter
\def\moverlay{\mathpalette\mov@rlay}
\def\mov@rlay#1#2{\leavevmode\vtop{   \baselineskip\z@skip \lineskiplimit-\maxdimen
   \ialign{\hfil$\m@th#1##$\hfil\cr#2\crcr}}}
\newcommand{\charfusion}[3][\mathord]{
    #1{\ifx#1\mathop\vphantom{#2}\fi
        \mathpalette\mov@rlay{#2\cr#3}
      }
    \ifx#1\mathop\expandafter\displaylimits\fi}
\makeatother

\DeclareFontFamily{U}  {MnSymbolC}{}
\DeclareSymbolFont{MnSyC}         {U}  {MnSymbolC}{m}{n}
\DeclareFontShape{U}{MnSymbolC}{m}{n}{
    <-6>  MnSymbolC5
   <6-7>  MnSymbolC6
   <7-8>  MnSymbolC7
   <8-9>  MnSymbolC8
   <9-10> MnSymbolC9
  <10-12> MnSymbolC10
  <12->   MnSymbolC12}{}
\DeclareMathSymbol{\powerset}{\mathord}{MnSyC}{180}

\usepackage{tikz}
\usetikzlibrary{calc,arrows,decorations.pathmorphing}
\pgfdeclarelayer{background}
\pgfdeclarelayer{foreground}
\pgfdeclarelayer{front}
\pgfsetlayers{background,main,foreground,front}

\usepackage{adjustbox}

\theoremstyle{plain}
\newtheorem{theorem}{Theorem}[section]		
\newtheorem{lemma}[theorem]{Lemma}
\newtheorem{claim}[theorem]{Claim}
\newtheorem{proposition}[theorem]{Proposition}

\newtheorem{conjecture}[theorem]{Conjecture}

\theoremstyle{remark}

\newcommand{\eps}{\ensuremath{\varepsilon}}
\def\Prob{\mathds{P}}
\def\E{\mathds{E}}
\def\N{\mathds{N}}

\def\R{\mathds{R}}
\def\CP{\mathcal{P}}

\def\Gr{G^{(r)}}
\def\eu{\mathrm{e}}

\let\originalleft\left
\let\originalright\right
\renewcommand{\left}{\mathopen{}\mathclose\bgroup\originalleft}
\renewcommand{\right}{\aftergroup\egroup\originalright}

\makeatletter
\def\imod#1{\allowbreak\mkern10mu({\operator@font mod}\,\,#1)}
\makeatother

\begin{document}
\title{Sharp thresholds for nonlinear Hamiltonian cycles in hypergraphs}

\author{Bhargav Narayanan}
\address{Department of Mathematics, Rutgers University, Piscataway, NJ 08854, USA}
\email{narayanan@math.rutgers.edu}

\author{Mathias Schacht}
\address{Department of Mathematics, Yale University, New Haven, USA}
\email{mathias.schacht@yale.edu}

\subjclass[2010]{Primary 05C80; Secondary 05C65, 05C45}

\begin{abstract}
	For positive integers $r > \ell$, an $r$-uniform hypergraph is called an \emph{$\ell$-cycle} if there exists a cyclic ordering of its vertices such that each of its edges consists of~$r$ consecutive vertices, and such that every pair of consecutive edges (in the natural ordering of the edges) intersect in precisely $\ell$ vertices. Such cycles are said to be \emph{linear} when~$\ell = 1$, and \emph{nonlinear} when $\ell > 1$. We determine the sharp threshold for nonlinear Hamiltonian cycles and show that for all $r > \ell > 1$, the threshold $p^*_{r, \ell} (n)$ for the appearance of a Hamiltonian $\ell$-cycle in the random $r$-uniform hypergraph on $n$ vertices is sharp and \mbox{is~$p^*_{r, \ell} (n) = \lambda(r,\l) (\frac{\eu}{n})^{r - \ell}$} for an explicitly specified function $\lambda$. 
	This resolves several questions raised by Dudek and Frieze in 2011.
\end{abstract}

\maketitle

\section{Introduction}
A basic problem in probabilistic combinatorics concerns locating the critical density at which a substructure of interest appears inside a random structure (with high probability). In the context of random graph theory, the question of when a random graph contains a \emph{Hamiltonian cycle} has received considerable attention.
Indeed, from the foundational works of P\'osa~\cite{ham1}, Koml\'os and Szemer\'edi~\cite{ham2}, Bollob\'as~\cite{ham3}, and Ajtai, Koml\'os and Szemer\'edi~\cite{ham4}, we have a very complete picture, understanding not only the \emph{sharp threshold} for this problem but the \emph{hitting time} as well. Since these early breakthroughs, there have been a number of papers locating thresholds, both asymptotic and sharp, for various spanning subgraphs of interest (see, e.g.,~\cites{oliver,richard} and the references therein for various related results).

In contrast, threshold results for spanning structures in the context of random hypergraph theory have been somewhat harder to come by. Indeed, even the basic question of locating the asymptotic threshold at which a random $r$-uniform hypergraph (or \emph{$r$-graph}, for short) contains a matching, i.e., a spanning collection of disjoint edges, proved to be a major challenge, resisting the efforts of a number of researchers up until the breakthrough work of Johansson, Kahn and Vu~\cite{jkv}; more recently, both the sharp threshold as well as the hitting time for this problem have been obtained by Kahn~\cite{jeff}. In the light of these developments for matchings, we study what is perhaps the next most natural question in this setting, namely that of when a random $r$-graph contains a Hamiltonian cycle; our main contribution is to resolve the sharp threshold problem for nonlinear Hamiltonian cycles.

There are multiple notions of cycles in hypergraphs, so let us recall the relevant definitions: given positive integers $r > \ell \ge 1$, an $r$-graph is called an \emph{$\ell$-cycle} if there exists a cyclic ordering of its vertices such that each of its edges consists of $r$ consecutive vertices in the ordering, and such that every pair of consecutive edges (in the natural ordering of the edges) intersect in precisely $\ell$ vertices (see Figure~\ref{pic:cycle} for an example). A \emph{Hamiltonian $\ell$-cycle} is then an $\ell$-cycle spanning the entire vertex set; of course, an $r$-graph on $n$ vertices may only contain a Hamiltonian $\ell$-cycle when $(r-\ell) \,|\, n$, and such a cycle then has precisely $n / (r - \ell)$ edges. Finally, by convention, an $\ell$-cycle is called \emph{linear} (or \emph{loose}) when $\ell = 1$, \emph{nonlinear} when $\ell > 1$, and \emph{tight} when $\ell = r-1$.

Given $r > \ell \ge 1$, we set
\[\lambda(r,\l) = t!\cdot(s-t)!,\]
where $s = r - \ell$ and $1 \le t \le s$ is the unique integer satisfying $t = r \imod{s}$, and define
\[ p^*_{r, \ell} (n) = \frac{\lambda(r,\l)\eu^s}{  n^{s}}.
\]
Writing $\Gr(n,p)$ for the binomial random $r$-graph on $n$ vertices, where each possible $r$-set of vertices appears as an edge independently with probability $p$, our main result is as follows.
\begin{theorem}\label{mainthm}
	For all integers $r > \ell > 1$ and all $\eps > 0$, as $n \to \infty$ with $(r-\ell) \,|\, n$, we have
	\[
		\Prob \left( \Gr(n,p) \text{ contains a Hamiltonian $\ell$-cycle} \right) \to
		\begin{cases}
			1 & \mbox{if } p > (1+ \eps)p^*, \text{ and} \\
			0 & \mbox{if } p < (1 -\eps)p^*,             \\
		\end{cases}
	\]
	where we abbreviate $p^* = p^*_{r, \ell} (n)$.
\end{theorem}

\begin{figure}
	\begin{center}
		\trimbox{0cm 0cm 0cm 0cm}{
			\begin{tikzpicture}[scale = 0.6]
				\foreach \x in {0,1,2,3,4,5,6,7}
				\node at (3*\x, 0) [inner sep=0.5mm, circle, fill=black!100] {};
				\foreach \x in {0,1,2,3,4,5,6}
				\node at (3*\x+1, 0) [inner sep=0.5mm, circle, fill=black!100] {};
				\foreach \x in {0,1,2,3,4,5,6}
				\node at (3*\x+2, 0) [inner sep=0.5mm, circle, fill=black!100] {};
				\foreach \x in {0,3,6,9,12,15}
				\draw[shift={(\x,0)}] (3,0) ellipse (3.5cm and 0.5cm);
				\draw[dotted, thick] (-1.5,0)--(-0.5,0);
				\draw[dotted, thick] (21.5,0)--(22.5,0);
			\end{tikzpicture}
		}
	\end{center}
	\caption{A $7$-uniform $4$-cycle.}\label{pic:cycle}
\end{figure}
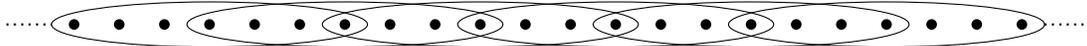

The critical density $p^*_{r, \ell}$ appearing in our result corresponds to the so-called `expectation threshold', namely the density above which the expected number of Hamiltonian $\ell$-cycles in~$\Gr(n,p)$ begins to diverge. A moment's thought should convince the reader that, unlike in the case of linear Hamiltonian cycles where one has to worry about isolated vertices, there are no `coupon collector type' obstacles to the presence of nonlinear Hamiltonian cycles; therefore, the conclusion of Theorem~\ref{mainthm} should not come as a surprise. Indeed, the problem of whether something like Theorem~\ref{mainthm} ought to hold was raised by Dudek and Frieze~\cites{loose, tight}. Towards such a result, they showed that $p^*_{r, \ell}$ is an asymptotic threshold for all nonlinear Hamiltonian cycles, that $p^*_{r, \ell}$ is a sharp threshold for tight Hamiltonian cycles when $r\ge 4$, and that $p^*_{r, \ell}$ is a semi-sharp threshold for all $r > \ell \ge 3$.

The main difficulty in proving Theorem~\ref{mainthm} is that, with the exception of the case of tight Hamiltonian cycles with $r \ge 4$ mentioned earlier, the second moment method is in itself not sufficient to prove the result; for instance, it is easy to verify, even in the simple case of $r = 3$ and $\ell = 2$ (i.e., tight Hamiltonian cycles in $3$-graphs) that the requisite second moment is too large to yield our result. To prove Theorem~\ref{mainthm}, we shall combine a careful second moment estimate, which necessitates working modulo various symmetries, with a powerful theorem of Friedgut~\cite{friedgut1} characterising coarse thresholds.

This paper is organised as follows. We gather the tools we require in Section~\ref{prelim}. The proof of Theorem~\ref{mainthm} follows in Section~\ref{proof}. We conclude in Section~\ref{conc} with a discussion of some open problems.
\section{Preliminaries}\label{prelim}
We begin with some background on thresholds. Recall that a \emph{monotone $r$-graph property~$W$} is a sequence $(W_n)_{n \ge 0}$ of families of $r$-graphs, where $W_n$ is a family of $r$-graphs on~$n$ vertices closed under the addition of edges and invariant under $r$-graph isomorphism.

Given a monotone $r$-graph property~$W = (W_n)_{n \ge 0}$,  a function $p^*(n)$ is said to be a \emph{threshold} or \emph{asymptotic threshold} for $W$  if $\Prob(\Gr(n,p) \in W_n)$ tends, as $n \to \infty$, either to $1$ or $0$ as $p / p^*$ tends either to $\infty$ or $0$ respectively, and a function $p^*(n)$ is said to be a \emph{sharp threshold} for  $W$ if $\Prob(\Gr(n,p) \in W_n)$ tends, as $n \to \infty$, either to $1$ or $0$ as $p / p^*$ remains bounded away from $1$ either from above or below respectively. Of course, thresholds and sharp thresholds are not unique, but following common practice, we will often say `the' threshold or sharp threshold when referring to the appropriate equivalence class of functions. Finally, a function $p^*(n)$ is said to be a \emph{semi-sharp threshold} for $W$ if there exist constants $C_0 \le 1 \le C_1$ such that $\Prob(\Gr(n,p) \in W_n)$ tends, as $n \to \infty$, either to $1$ or $0$ as~$p / p^*$ remains bounded below by $C_1$ or above by $C_0$ respectively; while we do not need this notion ourselves, we give this definition to place existing results around our main result in the appropriate context.

That every monotone property has an asymptotic threshold follows from a (much more general) result of Bollob\'as and Thomason~\cite{thresh}. Unlike with asymptotic thresholds, a monotone property need not necessarily have a sharp threshold; such properties are said to have \emph{coarse thresholds}. We shall make use of a powerful characterisation of monotone properties that have coarse thresholds due to Friedgut~\cite{friedgut1} which says, roughly, that such properties are `approximable by a local property'; a concrete formulation at a level of generality sufficient for our purposes, see~\cite{friedgut2}, is as follows.

\begin{proposition}\label{st}
	Fix $r \in \N$ and let $W = (W_n)_{n \ge 0}$ be a monotone $r$-graph property that has a coarse threshold. Then there exists a constant $\alpha > 0$, a threshold function $\hat p = \hat p(n)$ with
	\[ \alpha < \Prob\big(\Gr(n, \hat p) \in W_n\big) < 1-3\alpha \]
	for all $n \in \N$, a constant $\beta > 0$ and a fixed $r$-graph $F$ such that the following holds: for infinitely many $n \in \N$, there exists an $r$-graph on $n$ vertices $H_n \notin W_n$ such that
	\[ \Prob\big(H_n \cup \Gr(n,\beta \hat p) \in W_n\big) < 1-2\alpha, \]
	where the random $r$-graph $\Gr(n, \beta \hat p)$ is taken to be on the same vertex set as $H_n$, and
	\[ \Prob\big(H_n \cup \tilde F \in W_n\big) > 1-\alpha, \]
	where $\tilde F$ denotes a random copy of $F$ on the same vertex set as $H_n$. \qed
\end{proposition}

We shall also require the Paley--Zygmund inequality.
\begin{proposition}\label{pz}
	If $X$ is a non-negative random variable, then
	\[ \Prob(X > 0) \ge \frac{\E[X]^2}{\E[X^2]}. \eqno\qed \]
\end{proposition}

Finally, we collect together some standard estimates for factorials and binomial coefficients.
\begin{proposition}\label{stirling}
	For all $n \in \N$, we have
	\[ \sqrt{2\pi n} \left(\frac{n}{\eu}\right)^n \le n! \le \eu \sqrt{ n} \left(\frac{n}{\eu}\right)^n,\]
	and for all positive integers $1\le k \le n$, we have
	\[ \binom{n}{k} \le \left(\frac{\eu n}{k}\right)^k. \eqno \qed \]
\end{proposition}

\section{Proof of the main result}\label{proof}
In this section, we shall prove Theorem~\ref{mainthm}. We begin by setting up some notational conventions that we shall adhere to in the sequel.

In what follows, we fix $r, \ell \in \N$ with $r > \ell > 1$, set $s = r - \ell$, take $t$ to be the unique integer satisfying $t = r \imod{s}$ with $1 \le t \le s$, and set $\lambda = t!(s-t)!$. We shall henceforth assume that $n$ is a large integer divisible by $s$, and we set $m = n / s$ so that $m$ is the number of edges in an $\ell$-cycle on $n$ vertices. Finally, all $r$-graphs on $n$ vertices in the sequel will implicitly be assumed to be on the vertex set $[n] = \{1, 2, \dots, n\}$.

To deal with $r$-graph cycles on the vertex set $[n]$, we shall define an equivalence relation on $S_n$, the symmetric group of permutations of $[n] = \{1, 2,\dots, n\}$; we shall ignore the group structure of $S_n$ for the most part, so for us, a permutation $\sigma \in S_n$ is just an arrangement $\sigma(1), \sigma(2), \dots, \sigma(n)$ of the elements of $[n]$ (namely, vertices), at locations indexed by $[n]$.

We divide $[n]$ into $m$ \emph{blocks} of size $s$, where for $0\le i < m$, the $i$-th such block is comprised of the interval $\{is+1, is+2, \dots, is + s\}$ of vertices, and we further divide each such block into two \emph{subblocks}, where the $t$-subblock of a block consists of the first $t$ vertices in the block, and the $(s-t)$-subblock of a block consists of the last $s-t$ vertices in the block. Now, define an equivalence relation on $S_n$ by saying that two permutations $\sigma$ and $\tau$ are \emph{subblock equivalent} if $\tau$ may be obtained from $\sigma$ by only rearranging vertices within subblocks; in other words, an equivalence class of this equivalence relation may be viewed as an element of the quotient $Q_n = S_n / (S_t \times S_{s-t})^m$.

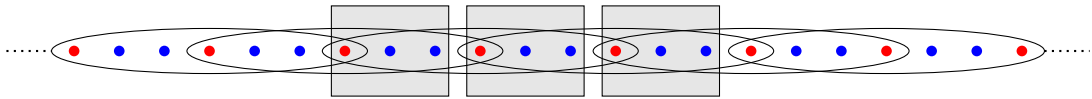
\begin{figure}
	\begin{center}
		\trimbox{0cm 0cm 0cm 0cm}{
			\begin{tikzpicture}[scale = 0.6]
				\draw[fill = gray!20] (5.7,-1) rectangle (8.3,1);
				\draw[shift={(3,0)},fill = gray!20] (5.7,-1) rectangle (8.3,1);
				\draw[shift={(6,0)}, fill = gray!20] (5.7,-1) rectangle (8.3,1);
				\foreach \x in {0,1,2,3,4,5,6,7}
				\node at (3*\x, 0) [inner sep=0.5mm, circle, fill=red!100] {};
				\foreach \x in {0,1,2,3,4,5,6}
				\node at (3*\x+1, 0) [inner sep=0.5mm, circle, fill=blue!100] {};
				\foreach \x in {0,1,2,3,4,5,6}
				\node at (3*\x+2, 0) [inner sep=0.5mm, circle, fill=blue!100] {};
				\foreach \x in {0,3,6,9,12,15}
				\draw[shift={(\x,0)}] (3,0) ellipse (3.5cm and 0.5cm);
				\draw[dotted, thick] (-1.5,0)--(-0.5,0);
				\draw[dotted, thick] (21.5,0)--(22.5,0);
			\end{tikzpicture}
		}
	\end{center}
	\caption{Blocks and subblocks of a $7$-uniform $4$-cycle.}\label{pic:blockstruc}
\end{figure}
The definition of the above equivalence relation is motivated by the natural $\ell$-cycle associated with a permutation: given $\sigma \in S_n$, consider the $r$-graph $H_\sigma$ on $[n]$ with $m$ edges, where for $0\le i < m$, the $i$-th edge of $H_\sigma$ is the $r$-set $\{\sigma(is+1), \sigma(is+2), \dots, \sigma(is + r)\}$, of course with indices being considered cyclically modulo $n$. It is easy to verify both that $H_\sigma$ is an $\ell$-cycle for each $\sigma \in S_n$, and that if $\sigma$ and $\tau$ are subblock equivalent, then $H_\sigma = H_\tau$; see Figure~\ref{pic:blockstruc} for an illustration. Hence, in what follows, we shall abuse notation and call the elements of $Q_n$ permutations (when strictly speaking, they are equivalence classes of permutations), and for $\sigma \in Q_n$, we write $H_\sigma$ for the natural $\ell$-cycle associated with $\sigma$.

We parameterise $p = Cp^*_{r,l}(n) = C \lambda \eu^s / n^s$ for some constant $C > 0$, and work with $G = \Gr(n,p)$, where we take the vertex set of $G$ to be $[n]$. Therefore, our goal is to show that $G$ contains a Hamiltonian $\ell$-cycle with high probability when $C > 1$ (namely, the $1$-statement), and that $G$ does not contain a Hamiltonian $\ell$-cycle with high probability when $C < 1$ (namely, the $0$-statement).

In what follows, constants suppressed by asymptotic notation are allowed to depend on fixed parameters (quantities depending only on $r$, $\ell$, $C$, etc.) but not on variables that depend on $n$, which we send to infinity along multiples of $s$.  We also adopt the standard convention that an event holds with high probability if the probability of the event in question is $1-o(1)$ as $n \to \infty$.

We shall focus our attention on the random variable $X$ that counts the number of $\sigma \in Q_n$ for which the $\ell$-cycle $H_\sigma$ is contained in $G$, noting that $G$ contains a Hamiltonian $\ell$-cycle if and only if $X > 0$.

We start by computing the first moment of $X$.

\begin{lemma}\label{expectation}
	We have $\E[X] = |Q_n|p^m = n! (p / \lambda)^m$, so that
	\[
		\E[X] \longrightarrow
		\begin{cases}
			\infty & \mbox{if } C > 1, \text{ and} \\
			0      & \mbox{if } C < 1.             \\
		\end{cases}
	\]
\end{lemma}
\begin{proof}
	This follows from noting that $|Q_n| = n! / \lambda^m$, estimating $n!$ using Proposition~\ref{stirling}, and using the fact that $n = ms$.
\end{proof}

In particular, the above first moment estimate, combined with Markov's inequality, establishes the $0$-statement. To establish the $1$-statement, the following second moment estimate will be crucial.

\begin{lemma}\label{variance}
	For $C > 1$, we have $\E[X^2] = O(\E[X]^2)$.
\end{lemma}

Let us point out that Lemma~\ref{variance} does not make the stronger promise that \[\E[X^2] = (1+o(1))\E[X]^2\,,\] and indeed, such an estimate does not hold generally for an arbitrary pair of integers $r > \ell > 1$.

\begin{proof}[Proof of Lemma~\ref{variance}]
	To estimate the second moment of $X$, it will be convenient to make the following definition: for $0 \le b \le m$, let $N(b)$ denote, for any fixed permutation $\sigma \in Q_n$, the number of permutations $\tau \in Q_n$ meeting $\sigma$ in $b$ edges, by which we mean that $H_\tau$ intersects $H_\sigma$ in exactly $b$ edges. With this definition in place, using the trivial fact that $N(0) \le |Q_n|$, we have
	\begin{align*}
		\E[X^2] & = \sum_{\sigma, \tau \in Q_n} \Prob (H_\sigma \cup H_\tau \subset G) \\
		        & = |Q_n|p^m \sum_{b = 0}^m N(b) p^{m-b}                               \\
		        & \le |Q_n|^2p^{2m} + |Q_n|p^m \sum_{b = 1}^m N(b) p^{m-b}             \\
		        & =\E[X]^2 + \E[X]\sum_{b = 1}^m N(b) p^{m-b},
	\end{align*}
	whence it follows that
	\[
		\frac{\E[X^2] }{ \E[X] ^2} \le 1 + \sum_{b = 1}^m \frac{N(b)p^{-b}} {|Q_n|},
	\]
	so to prove Lemma~\ref{variance}, it suffices to show that the sum
	\[\Gamma = \sum_{b = 1}^m \frac{N(b)p^{-b}}{ |Q_n|} \]
	satisfies the estimate $\Gamma = O(1)$ when $C > 1$.

	The rough plan of attack now is similar to that adopted by Dudek and Frieze~\cite{tight}, but we shall require a more careful two-stage analysis since we require stronger estimates: first, we shall control the `canonical' contributions to $\Gamma$, and subsequently bound the `non-canonical' contributions in terms of the aforementioned `canonical' ones; we make precise these notions below.

	An $r$-graph is called an \emph{$\ell$-path} if there exists a linear ordering of its vertices such that each of its edges consists of $r$ consecutive vertices, and such that every pair of consecutive edges (in the natural ordering of the edges) intersect in precisely $\ell$ vertices. Given a permutation $\sigma \in Q_n$, we say that $\tau \in Q_n$ meets $\sigma$ \emph{canonically} if $H_\tau$ meets $H_\sigma$ in a family of vertex-disjoint $\ell$-paths, and we otherwise say that $\tau$ meets $\sigma$ \emph{non-canonically}.

	For $1 \le b \le m$, let $N_c(b)$ to be the number of permutations $\tau \in Q_n$ which canonically meet a fixed permutation $\sigma \in Q_n$ in $b$ edges, set $N'(b) = N(b) - N_c(b)$, and decompose $\Gamma = \Gamma_c + \Gamma'$, where naturally
	\[\Gamma_c  =  \sum_{b = 1}^m \frac{N_c(b)p^{-b}} {|Q_n|}\] and
	\[\Gamma'  = \Gamma - \Gamma_c = \sum_{b = 1}^m \frac{N'(b)p^{-b} }{ |Q_n|}.\]

	First, we bound the canonical contributions to $\Gamma$.

	\begin{claim}\label{canonical}
		For $C > 1$, we have $\Gamma_c = O(1)$.
	\end{claim}
	\begin{proof}
		Fix a permutation $\sigma \in Q_n$ and for $1 \le a \le b$, write $N_c (b, a)$ for the number of permutations $\tau \in Q_n$ which meet $\sigma$ canonically in $b$ edges which together form $a$ vertex-disjoint $\ell$-paths in $H_\sigma$. We now proceed to estimate $N_c (b,a)$.

		Given $\sigma$, a $(b,a)$-configuration in $\sigma$ is a collection of $b$ edges in~$H_\sigma$ which together form~$a$ vertex-disjoint $\ell$-paths; clearly, a $(b,a)$-configuration covers $sb + \ell a$ vertices. The number of ways to choose a $(b,a)$-configuration in $\sigma$ is clearly at most \[\binom{m}{a} \binom{b}{a},\] since there are at most $\binom{m}{a}$ ways of locating the leftmost edge in each of the $a$ $\ell$-paths in~$H_\sigma$, and the number of ways to subsequently choose the number of edges in each of these $a$ paths so that there are $b$ edges in total is clearly at most the number of solutions to the equation $x_1 + x_2 + \dots + x_a = b$ over the positive integers, which is $\binom{b-1}{a-1} \le \binom{b}{a}$.

		Next, given a $(b,a)$-configuration $\CP$ in $\sigma$, let us count the number of choices for $\tau \in Q_n$ for which $H_\tau$ contains $\CP$. We do this in two steps. First, we count the number of ways in which the vertices covered by $\CP$ can be embedded into $\tau$, and then estimate the number of ways in which the vertices not covered by $\CP$ can be ordered in $\tau$, ensuring at all times that we only count up to subblock equivalence.

		\begin{figure}
			\begin{center}
				\trimbox{0cm 0cm 0cm 0cm}{
					\begin{tikzpicture}[scale = 0.6]
						\draw[thick,right hook-latex] (0,1.5) --(4,1.5);
						\draw[fill = gray!20] (5.7,-1) rectangle (14.3,1);
						\foreach \x in {0,1,2,3,4,5,6,7}
						\node at (3*\x, 0) [inner sep=0.5mm, circle, fill=red!100] {};
						\foreach \x in {0,1,2,3,4,5,6}
						\node at (3*\x+1, 0) [inner sep=0.5mm, circle, fill=blue!100] {};
						\foreach \x in {0,1,2,3,4,5,6}
						\node at (3*\x+2, 0) [inner sep=0.5mm, circle, fill=blue!100] {};
						\foreach \x in {0,3,6,9,12,15}
						\draw[shift={(\x,0)}] (3,0) ellipse (3.5cm and 0.5cm);

						\draw[shift={(0,-5)},thick,left hook-latex] (21,1.5) --(17,1.5);
						\draw[shift={(1,-5)},fill = gray!20] (5.7,-1) rectangle (14.3,1);
						\foreach \x in {0,1,2,3,4,5,6,7}
						\node at (3*\x, -5) [inner sep=0.5mm, circle, fill=red!100] {};
						\foreach \x in {0,1,2,3,4,5,6}
						\node at (3*\x+1, -5) [inner sep=0.5mm, circle, fill=blue!100] {};
						\foreach \x in {0,1,2,3,4,5,6}
						\node at (3*\x+2, -5) [inner sep=0.5mm, circle, fill=blue!100] {};
						\foreach \x in {0,3,6,9,12,15}
						\draw[shift={(\x,-5)}] (3,0) ellipse (3.5cm and 0.5cm);
					\end{tikzpicture}
				}
			\end{center}
			\caption{The rigid interior of a $7$-uniform $4$-path in the two possible directions of embedding.}\label{pic:rigid}
		\end{figure}
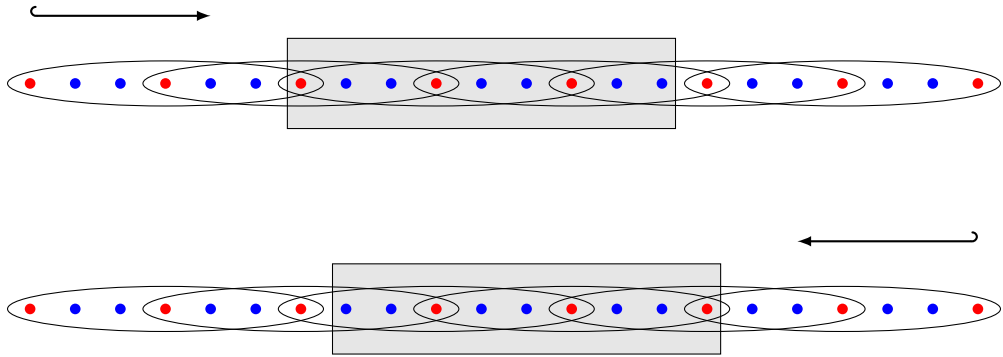
		Now, there are at most $a! \binom{m}{a}$ ways to choose the starting blocks of the leftmost edges of the $a$ distinct $\ell$-paths of $\CP$ in $\tau$. Once the left endpoint of one of these $\ell$-paths has been fixed in $\tau$, we observe that there are only $O(1)$ ways, up to subblock equivalence, to embed the remaining vertices of this $\ell$-path into $\tau$; indeed, the relative ordering of all the vertices in an $\ell$-path, with $O(1)$ exceptions at the left and right extremes, is rigid up to subblock equivalence, up to a reversal of the direction of embedding (left-to-right or right-to-left), as shown in Figure~\ref{pic:rigid}. Consequently, once the location of the $a$ leftmost edges have been determined in $\tau$, the number of ways of embedding the rest of $\CP$ into $\tau$ is at most $L^a$ for some $L = L(r,\ell)$. We conclude that the number of ways to embed $\CP$ is at most
		\[ \binom{m}{a} a! L^a.\]

		Once we have embedded $\CP$ into $\tau$, there are $(n-sb-\ell a)!$ ways to arrange the remaining vertices uncovered by $\CP$, without accounting for subblock equivalence. It is easy to check that any embedding of $\CP$ covers at most $b + \ell a$ blocks in $\tau$, so the number of choices for~$\tau \in Q_n$ with a given embedding of $\CP$ is at most
		\[ \frac{(n-sb-\ell a)!}{\lambda^{m-b-\ell a}}. \]

		From the above estimates and using Proposition~\ref{stirling} to bound binomial coefficients, we conclude that
		\begin{align}\label{canoncount}
			N_c (b,a) & \le \binom{m}{a} \binom{b}{a} \binom{m}{a} a! L^a \frac{(n-sb-\ell a)!}{\lambda^{m-b-\ell a}} \nonumber \\
			          & \le \exp{(O(a))} \frac{n^{2a}b^a}{a^{2a}} \frac{(n-sb-\ell a)!}{\lambda^{m-b}},
		\end{align}
		where, as remarked upon before, constants suppressed by the asymptotic notation depend only on $r$ and $\ell$.

		To finish the proof of the lemma, we now use, in order, the above bound~\eqref{canoncount}, the fact that $\ell \ge 2$, Proposition~\ref{stirling}, and the fact that $1+x \le \eu^x$ for all $x \in \R$ to show that
		\begin{align*}
			\Gamma_c & =  \sum_{b = 1}^m \sum_{a = 1}^b \frac{N_c(b,a)p^{-b}}{|Q_n|}                                                                                                         \\
			         & \le \sum_{b = 1}^m \sum_{a = 1}^b \exp{(O(a))} \frac{n^{2a}b^a}{a^{2a}} \frac{(n-sb-\ell a)! }{\lambda^{m-b}} \frac{\lambda^m}{n!} \frac{n^{sb}}{C^b\lambda^b \eu^{sb}} \\
			         & \le  \sum_{b = 1}^m \sum_{a = 1}^b C^{-b}\exp{(O(a))} \frac{n^{2a}b^a}{a^{2a}} \frac{(n-sb-2a)! }{n!} \frac{n^{sb}}{ \eu^{sb}}                                          \\
			         & \le  \sum_{b = 1}^m \sum_{a = 1}^b C^{-b}\exp{(O(a))} \frac{n^{sb + 2a}b^a}{a^{2a}} \frac{(n-sb-2a)^{n -sb-2a} }{n^n}                                                 \\
			         & \le  \sum_{b = 1}^m  \sum_{a = 1}^b C^{-b}\exp{(O(a))} \frac{b^a}{a^{2a}} \left( 1 -  \frac{sb+2a}{n} \right)^{n -sb-2a}                                              \\
			         & \le  \sum_{b = 1}^m \sum_{a = 1}^b C^{-b} \exp{(a\log b - 2a \log a -sb -2a + (sb + 2a)^2/n + O(a) )}.
		\end{align*}
		We uniformly have $(4a^2 + 4sab)/n = O(n)$, since $a \le b \le m \le n$, so the above estimate reduces to
		\begin{equation}\label{est1}
			\Gamma_c \le  \sum_{b = 1}^m \left( C^{-b} \exp{(-sb + (sb)^2/n)} \sum_{a = 1}^b \exp{(a\log b - 2a \log a + O(a) )}\right).
		\end{equation}
		Finally, since $sb \le sm = n$, we uniformly have \[ \exp{(-sb + (sb)^2/n)} \le 1\] for all $1 \le b \le m$, and it is straightforward to verify that we uniformly have
		\[\exp{(a\log b - 2a \log a + O(a) )} = \exp(o(b))\] for all $1 \le a \le b$. Using these two bounds in~\eqref{est1}, it follows that for $C > 1$, we have
		\[ \Gamma_c \le \sum_{b = 1}^m  C^{-b} b\exp(o(b)) = \sum_{b = 1}^m C^{-b + o(b)} = O(1), \]
		proving the claim.
	\end{proof}

	The second and final step in the proof of Lemma~\ref{variance} is to estimate the non-canonical contributions to $\Gamma$.
	\begin{claim}\label{non-canonical}
		For $C > 1$, we have $\Gamma' = O(1)$.
	\end{claim}
	\begin{proof}
		We shall prove the claim by means of a comparison argument: we shall demonstrate how we may group summands in $\Gamma'$ so as to get estimates analogous to those that we obtained for $\Gamma$ in the proof of Claim~\ref{canonical}.

		For any $\sigma, \tau \in Q_n$, we may decompose the intersection of $H_\sigma$ and $H_\tau$ into a collection of vertex-disjoint \emph{weak paths}, where a \emph{weak path} is a just a sequence of edges in which every consecutive pair of edges intersect.

		We fix a permutation $\sigma \in Q_n$ for the rest of the argument. Given a weak path $P'$ in $H_\sigma$, notice that there is a unique minimal $\ell$-path $P$ in $H_\sigma$ covering precisely the same set of vertices as $P'$; we call $P$ the \emph{minimal cover} of $P'$. Now, given any $\tau \in Q_n$, there is a unique minimal \emph{covering configuration in $\sigma$} associated with $\tau$ obtained by taking the minimal covers of each of the weak paths in which $H_\tau$ meets $H_\sigma$. To prove the claim, we shall show that the contributions to $\Gamma'$ from all those $\tau \in Q_n$ whose covering configuration is $\CP$ is comparable to the contributions to $\Gamma_c$ from all those $\tau \in Q_n$ meeting $\sigma$ canonically in $\CP$.

		We fix a $(b,a)$ configuration $\CP$ in $\sigma$ consisting of $b$ edges in total distributed across $a$ $\ell$-paths, and we consider the set of permutations $\tau \in Q_n$ with minimal cover $\CP$ that meet $\sigma$ non-canonically; we additionally parametrise this set by $1 \le k \le b$, writing $Q(\CP, k)$ for the set of such permutations $\tau$ for which there are $k$ edges of $\CP$ missing from the intersection~$\CP'$ of $H_\tau$ and $H_\sigma$.

		We claim that the number of ways to select a configuration $\CP'$ as above, and then embed the vertices covered by $\CP'$ into a permutation $\tau \in Q(\CP, k)$ in such a way that $\CP'$ is contained in $H_\tau$ is, up to subblock equivalence, at most
		\[\binom{b}{k} \binom{m}{a}a! R^{a + k}\]
		for some $R = R(r, \ell)$.

		We may verify the estimate above as follows. The number of possible choices for $\CP'$, namely the number of ways to choose $k$ edges from $\CP$ such that each of the $\ell$-paths of~$\CP$ remains a weak path after these $k$ edges are removed, may be crudely bounded above by $\binom{b}{k}$. As in the proof of Claim~\ref{canonical}, there are $\binom{m}{a}a!$ ways to choose the starting blocks of the leftmost edges of the $a$ distinct weak paths of $\CP'$ in $\tau$. Assume now that we have fixed $\mathcal{P}'$, the starting blocks in $\tau$ of the leftmost edges of the $a$ weak paths of $\CP'$, and the directions of embedding of these weak paths into $\tau$ (for which there are $2^a$ choices). Now, it is easy to see from the linear structure of a weak path that the relative order of vertices in disjoint edges of a weak path must be preserved in any embedding of that weak path into $\tau$, so in particular, there are only $O(1)$ choices for the location in $\tau$ of any particular vertex covered by $\CP'$ (once endpoints and directions of embedding have been fixed, as we have assumed). Furthermore, it follows from the rigidity of an $\ell$-path (as in the proof of Claim~\ref{canonical}) that any vertex covered by $\mathcal{P}'$, with $O(1)$ exceptions at the left and right extremities of each of the $a$ weak paths, which possesses potential embedding locations in more than one subblock must necessarily be within $O(1)$ distance (in $\sigma$) of some edge present in $\CP$ but not in $\CP'$; clearly there $O(k)$ such vertices in total. These facts taken together demonstrate the validity of the bound claimed above.

		Now, noting that the contribution of any $\tau \in Q(\CP, k)$ to $\Gamma'$ is a factor of $p^k$ times the contribution to $\Gamma_c$ from any $\tau \in Q_n$ meeting
		$\sigma$ canonically in $\CP$, we may mimic the proof of Claim~\ref{canonical} to show that
		\[\Gamma' \le \sum_{b = 1}^m \sum_{a = 1}^b \left(\exp{(O(a))} \frac{n^{2a}b^a}{a^{2a}} \frac{(n-sb-\ell a)! }{\lambda^{m-b}} \frac{\lambda^m}{n!} \frac{n^{sb}}{C^b\lambda^b \eu^{sb}} \sum_{k=1}^b\binom{b}{k}R^kp^k \right). \]
		Observing that
		\[\sum_{k=1}^b\binom{b}{k}R^kp^k = \sum_{k=1}^b \exp(O(k)) \frac{b^k}{k^kn^{sk}} = O(1), \]
		we are left with an estimate for $\Gamma'$ of the same form as the one for $\Gamma'$ which we showed to be $O(1)$ when $C>1$ in the proof of Claim~\ref{canonical}; the claim follows.
	\end{proof}
	The two claims above together imply that $\Gamma = O(1)$ when $C > 1$, from which it follows that $\E[X^2] = O(\E[X]^2)$ when $C > 1$; the result follows.
\end{proof}

With our moment estimates in hand, we are now ready to prove our main result.

\begin{proof}[Proof of Theorem~\ref{mainthm}]
	As mentioned earlier, the $0$-statement, namely that $\Gr(n,p)$ does not contain a Hamiltonian $\ell$-cycle with high probability if $p < (1 + \eps)p^*_{r,\ell} (n)$ follows immediately from Lemma~\ref{expectation} and Markov's inequality.

	We prove the $1$-statement, namely that $\Gr(n,p)$ contains a Hamiltonian $\ell$-cycle with high probability if $p > (1 + \eps)p^*_{r,\ell} (n)$, by showing that the property of containing a Hamiltonian $\ell$-cycle has a sharp threshold, and that this sharp threshold must (asymptotically) necessarily be $p^*_{r,\ell}(n)$.

	If $p > (1+\eps)p^*_{r,\ell} (n)$, then it follows from Lemma~\ref{variance} and the Paley--Zygmund inequality, i.e., Proposition~\ref{pz}, that $\Gr(n,p)$ contains a Hamiltonian $\ell$-cycle with probability at least $\delta>0$ for some $\delta = \delta(\eps, r, \ell)$; consequently, if the property of containing a Hamiltonian $\ell$-cycle has a sharp threshold, this sharp threshold is necessarily asymptotic to $p^*_{r, \ell}(n)$.

	It remains to prove that the monotone $r$-graph property $W = (W_n)_{n \ge 0}$ of containing a Hamiltonian $\ell$-cycle has a sharp threshold, so suppose for the sake of a contradiction that~$W$ has a coarse threshold.

	It follows from Proposition~\ref{st} that there is a fixed $r$-graph $F$ and a threshold function $\hat p = \hat p(n)$ with the property that for infinitely many $n \in \N$, there is an $r$-graph $H_n \notin W_n$ on $n$ vertices such that adding a random copy of $F$ to $H_n$ is significantly more likely to make the resulting graph contain a Hamiltonian $\ell$-cycle than adding a random collection of edges of density about $\hat p$; concretely, for some universal constants $\alpha, \beta > 0$, we have
	\begin{equation}\label{notboost} \Prob(H_n \cup \Gr(n,\beta \hat p) \in W_n ) < 1-2\alpha, \end{equation}
	where the random $r$-graph $\Gr(n, \beta \hat p)$ is taken to be on the same vertex set as $H_n$, and
	\begin{equation}\label{boost} \Prob(H_n \cup \tilde F  \in W_n ) > 1-\alpha, \end{equation}
	where $\tilde F$ denotes a random copy of $F$ on the same vertex set as $H_n$.

	Now, the only way $F$ can help induce a Hamiltonian $\ell$-cycle in $H_n$ is through some sub-hypergraph of itself that appears in all large enough $\ell$-cycles, so by pigeonholing (and adding extra edges if necessary), we conclude from~\eqref{boost} that there exists a fixed $\ell$-path $P$, say with $k$ edges on $\ell + sk$ vertices, with the property that, for some universal constant $\gamma > 0$, we have
	\[ \Prob(H_n \cup \tilde P  \in W_n ) > \gamma, \]
	where $\tilde P$ again denotes a random copy of $P$ on the same vertex set as $H_n$. In other words, a positive fraction of all the possible ways to embed $P$ into the vertex set of $H_n$ are \emph{useful} and end up completing Hamiltonian $\ell$-cycles.

	Since $\hat p$ is an asymptotic threshold for $W$, clearly $\hat p(n) = \Theta(p^*_{r,\ell} (n)) = \Theta(n^{-s})$, since $p^*_{r, \ell}$ is also an asymptotic threshold for $W$, as can be read off from the proof of Lemma~\ref{variance}. On the other hand, the expected number of useful copies of $P$ created by the addition of a~$\beta \hat p = \Theta(n^{-s})$ density of random edges to $H_n$ is
	\[\Omega\left(\binom{n}{\ell + sk}(n^{-s})^k\right) = \Omega\left(n^\ell\right),
	\] and a routine application of the second moment method (indeed, $\ell$-paths are suitably `balanced') shows that adding a $\beta \hat p$ density of random edges to $H_n$ must, with high probability, create at least one useful copy of $P$ in $H_n$ and complete a Hamiltonian $\ell$-cycle, contradicting~\eqref{notboost}.

	We have now shown that $W$ has a sharp threshold, and that this threshold must be asymptotic to $p^*_{r,\ell}(n)$; the $1$-statement follows, completing the proof.
\end{proof}
\section{Conclusion}\label{conc}
There are two basic questions that our work raises; we conclude this paper by discussing these problems.

First, now that we have identified the sharp threshold for the appearance of nonlinear Hamiltonian cycles, one can and should ask about the `width' of the critical window. Since the sharp threshold corresponds to the expectation threshold, we do not expect the hitting time to be of interest. Nonetheless, it is plausible that the expectation threshold is much sharper than what we have shown, and we conjecture the following.
\begin{conjecture}\label{window}
	For all integers $r > \ell > 1$, if $p = p(n)$ is such that, as $n \to \infty$, we have $\E[X_\ell] \to \infty$, then
	\[ \Prob \left( \Gr(n,p) \text{ contains a Hamiltonian $\ell$-cycle} \right) \to 1, \]
	where $X_\ell$ is the random variable counting the number of Hamiltonian $\ell$-cycles in $\Gr(n,p)$.
\end{conjecture}

Second, it is natural to ask what happens for linear Hamiltonian cycles. The proof of Theorem~\ref{mainthm} shows that the appearance of a linear Hamiltonian cycle in $\Gr(n,p)$ has a sharp threshold, and we expect this sharp threshold to coincide with the sharp threshold for the disappearance of isolated vertices (i.e., vertices not contained in any edges). For $r \ge 3$, writing
\[p^{\mathrm{deg}}_r(n) = \frac{(r-1)!\log n} { n^{r-1}}\] to denote the sharp threshold for the disappearance of isolated vertices in $\Gr(n,p)$, we predict the following.

\begin{conjecture}\label{linear}
	For each $r \ge 3$, $p^{\mathrm{deg}}_r(n)$ is the sharp threshold for the appearance of a linear Hamiltonian cycle in $\Gr(n,p)$.
\end{conjecture}

In the case where $r=3$, Frieze~\cite{alan} showed that $p^{\mathrm{deg}}_3$ is a semi-sharp threshold for the appearance of a linear Hamiltonian cycle, and Dudek and Frieze~\cite{loose} showed that $p^{\mathrm{deg}}_r$ is an asymptotic threshold for the appearance of a linear Hamiltonian cycle for all $r\ge 3$. Of course, we expect much more than Conjecture~\ref{linear} to be true and naturally expect the hitting time for the appearance of a linear Hamiltonian cycle to coincide with the hitting time for the disappearance of isolated vertices, but even Conjecture~\ref{linear} appears to be out of the reach of existing techniques.

\section*{Acknowledgements}
The first author wishes to acknowledge support from NSF grant DMS-1800521.


\begin{bibdiv}
\begin{biblist}

\bib{ham4}{article}{
   author={Ajtai, M.},
   author={Koml\'os, J.},
   author={Szemer\'edi, E.},
   title={First occurrence of Hamilton cycles in random graphs},
   conference={
      title={Cycles in graphs},
      address={Burnaby, B.C.},
      date={1982},
   },
   book={
      series={North-Holland Math. Stud.},
      volume={115},
      publisher={North-Holland, Amsterdam},
   },
   date={1985},
   pages={173--178},
   review={\MR{821516}},
}

\bib{ham3}{article}{
   author={Bollob\'as, B\'ela},
   title={The evolution of sparse graphs},
   conference={
      title={Graph theory and combinatorics},
      address={Cambridge},
      date={1983},
   },
   book={
      publisher={Academic Press, London},
   },
   date={1984},
   pages={35--57},
   review={\MR{777163}},
}

\bib{thresh}{article}{
   author={Bollob\'{a}s, B.},
   author={Thomason, A.},
   title={Threshold functions},
   journal={Combinatorica},
   volume={7},
   date={1987},
   number={1},
   pages={35--38},
   issn={0209-9683},
   review={\MR{905149}},
   doi={10.1007/BF02579198},
}

\bib{tight}{article}{
   author={Dudek, Andrzej},
   author={Frieze, Alan},
   title={Tight Hamilton cycles in random uniform hypergraphs},
   journal={Random Structures Algorithms},
   volume={42},
   date={2013},
   number={3},
   pages={374--385},
   issn={1042-9832},
   review={\MR{3039684}},
   doi={10.1002/rsa.20404},
}
		

\bib{loose}{article}{
   author={Dudek, Andrzej},
   author={Frieze, Alan},
   title={Loose Hamilton cycles in random uniform hypergraphs},
   journal={Electron. J. Combin.},
   volume={18},
   date={2011},
   number={1},
   pages={Paper 48, 14},
   issn={1077-8926},
   review={\MR{2776824}},
}		

\bib{friedgut1}{article}{
   author={Friedgut, Ehud},
   title={Sharp thresholds of graph properties, and the $k$-sat problem},
   note={With an appendix by Jean Bourgain},
   journal={J. Amer. Math. Soc.},
   volume={12},
   date={1999},
   number={4},
   pages={1017--1054},
   issn={0894-0347},
   review={\MR{1678031}},
   doi={10.1090/S0894-0347-99-00305-7},
}

\bib{friedgut2}{article}{
   author={Friedgut, Ehud},
   title={Hunting for sharp thresholds},
   journal={Random Structures Algorithms},
   volume={26},
   date={2005},
   number={1-2},
   pages={37--51},
   issn={1042-9832},
   review={\MR{2116574}},
   doi={10.1002/rsa.20042},
}

\bib{alan}{article}{
   author={Frieze, Alan},
   title={Loose Hamilton cycles in random 3-uniform hypergraphs},
   journal={Electron. J. Combin.},
   volume={17},
   date={2010},
   number={1},
   pages={Note 28, 4},
   issn={1077-8926},
   review={\MR{2651737}},
}







\bib{jkv}{article}{
   author={Johansson, Anders},
   author={Kahn, Jeff},
   author={Vu, Van},
   title={Factors in random graphs},
   journal={Random Structures Algorithms},
   volume={33},
   date={2008},
   number={1},
   pages={1--28},
   issn={1042-9832},
   review={\MR{2428975}},
   doi={10.1002/rsa.20224},
}

\bib{jeff}{misc}{
	author={Kahn, J.},
	note={Personal communication},
	date={February 2018},
	
}

\bib{ham2}{article}{
   author={Koml\'{o}s, J\'{a}nos},
   author={Szemer\'{e}di, Endre},
   title={Limit distribution for the existence of Hamiltonian cycles in a
   random graph},
   journal={Discrete Math.},
   volume={43},
   date={1983},
   number={1},
   pages={55--63},
   issn={0012-365X},
   review={\MR{680304}},
   doi={10.1016/0012-365X(83)90021-3},
}

\bib{richard}{article}{
   author={Montgomery, R.},
   title={Spanning trees in random graphs},
   note={Submitted},
   eprint={1810.03299},
}


\bib{ham1}{article}{
   author={P\'osa, L.},
   title={Hamiltonian circuits in random graphs},
   journal={Discrete Math.},
   volume={14},
   date={1976},
   number={4},
   pages={359--364},
   issn={0012-365X},
   review={\MR{0389666}},
}

\bib{oliver}{article}{
   author={Riordan, Oliver},
   title={Spanning subgraphs of random graphs},
   journal={Combin. Probab. Comput.},
   volume={9},
   date={2000},
   number={2},
   pages={125--148},
   issn={0963-5483},
   review={\MR{1762785}},
}


\end{biblist}
\end{bibdiv}
\end{document}